\documentclass[oneside,reqno]{amsart}
\usepackage{amsmath, amsthm, mathrsfs, amssymb, amsfonts, mathtools}
\usepackage{graphicx}
\usepackage{hyperref}
\hypersetup{colorlinks=true, linkcolor=blue, citecolor=magenta, filecolor=magenta, urlcolor=cyan}

\usepackage{doi}

\usepackage[backend=bibtex, style=trad-alpha, doi=true, isbn=false, hyperref=true, backref=false, giveninits=true, abbreviate=true, sorting=nyt]{biblatex}
\usepackage[dvipsnames]{xcolor}

\newtheorem{theorem}{Theorem}
\newtheorem{remark}{Remark}

\definecolor{verylightgray}{rgb}{0.9,0.9,0.9}

\def\rhoOne{\rho_{\tau}}
\def\rhoTwo{\rho_{x}}

\def\Yfin{Y}

\newtheorem{lemma}{Lemma}

\usepackage{orcidlink}

\usepackage{afterpage}

\numberwithin{equation}{section}
\mathtoolsset{showonlyrefs}

\newcommand{\codett}[1]{\texttt{\detokenize{#1}}}

\begin{document}

\title{New class of time-periodic solutions to the 1D cubic wave equation}

\author{Filip Ficek\,\orcidlink{0000-0001-5885-7064}}
\address{(F.F.) University of Vienna, Faculty of Mathematics, Oskar-Morgenstern-Platz 1, 1090 Vienna, Austria and University of Vienna, Gravitational Physics, Boltzmanngasse 5, 1090 Vienna, Austria}
\email[]{filip.ficek@univie.ac.at}
\author{Maciej Maliborski\,\orcidlink{0000-0002-8621-9761}}
\address{(M.M.) University of Vienna, Faculty of Mathematics, Oskar-Morgenstern-Platz 1, 1090 Vienna, Austria and University of Vienna, Gravitational Physics, Boltzmanngasse 5, 1090 Vienna, Austria}
\email[]{maciej.maliborski@univie.ac.at}
\thanks{We thank Roland Donninger for valuable discussions.
We acknowledge the support of the Austrian Science Fund (FWF) through Project \href{http://doi.org/10.55776/P36455}{P~36455}, Wittgenstein Award \href{http://doi.org/10.55776/Z387}{Z~387}, and the START-Project \href{http://doi.org/10.55776/Y963}{Y~963}. Computations were performed on a local cluster supported by the Gravitational Physics Group at the University of Vienna.}


\begin{abstract}
In recent papers \cite{Ficek.2024A, Ficek.2024B} we presented results suggesting the existence of a new class of time-periodic solutions to the defocusing cubic wave equation on a one-dimensional interval with Dirichlet boundary conditions. 
Here we confirm these findings by rigorously constructing solutions from this class.
The proof uses rational arithmetic computations to verify essential operator bounds.
\end{abstract}

\date{\today}

\maketitle

\tableofcontents

\section{Introduction}
\label{sec:Introduction}

Time-periodic solutions to the cubic wave equation on a one-dimensional interval with Dirichlet boundary conditions
\begin{equation}
    \label{eq:nlw0}
    \partial_{t}^{2}u - \partial_{x}^{2}u + u^{3} = 0\,,
    \quad
    u(t,0) = 0 = u(t,\pi)\,,
\end{equation} 
have been an object of thorough investigations since the classic paper \cite{lidskii1988periodic}. The following years brought many additional existence results obtained via different methods \cite{Berti.2006, berti2007nonlinear, berti2008cantor, GM.2004, GMP.2005}. The common feature shared by these works is that the studied solutions form a family that can be described as a Cantor-like set bifurcating from the solution of the linearised problem, see Fig.~2 in \cite{wayne1997periodic} for a simplified picture. The apparent gaps in this family come from the small divisor problem \cite{arnold1988dynamical}. In \cite{Ficek.2024A} we presented numerical results suggesting a much more complicated structure of the solutions, that was further investigated in \cite{Ficek.2024B}. The aforementioned family (together with its rescalings, see Eq.~(2.2) in \cite{Ficek.2024A}) forms just a part of this structure, that we have called the \textit{trunk}. In addition to it, there also exist solutions characterised by larger energies and a more notable presence of the higher modes. They form \textit{branches} emerging from the trunk (see Fig.~\ref{fig:EnergyNormFrequency}).

The approach employed in \cite{Ficek.2024A} was based on the Galerkin method, letting us approximate Eq.~\eqref{eq:nlw0} by finite systems of nonlinear algebraic equations. The resulting systems could then be solved numerically, leading to the discovery of the intricate patterns. By increasing the truncation in the Galerkin method, we could notice that the observed patterns remain, but also new structures start emerging from them. This led us to the conjecture that in the limit of the PDE, one is left with solutions forming an infinitely complex fractal-like structure.
In this paper, we prove the existence of three distinct solutions to Eq.~\eqref{eq:nlw0}, characterised by the same frequency, that are close (with respect to a suitable norm) to the respective approximate solutions found in \cite{Ficek.2024A}. 
One of them belongs to the family earlier covered by the literature; it lies on the trunk and is dominated by the lowest mode. The remaining two have larger energies and a more complicated modes composition. They are parts of a branch, therefore supporting the conjectures put forward in~\cite{Ficek.2024A} and~\cite{Ficek.2024B}.

Our strategy is similar to the one used in \cite{Arioli.2017}. Its main step consists of estimating bounds on certain operators. Due to the sheer number of calculations required for this, the assistance of a computer comes in handy. In \cite{Arioli.2017} the authors made use of it by applying interval arithmetic, while we perform all of the calculations on rational numbers, see for example \cite{donninger2024self}, which guarantees the computations are free from rounding errors.
In addition to this technical detail, the main difference between this work and \cite{Arioli.2017} is that our focus is on solutions that are not dominated by the lowest mode, since they shall be examples of the new class of solutions conjectured in \cite{Ficek.2024A,Ficek.2024B}.

As a first step, let us introduce the rescaled time $\tau=\Omega t$ so Eq.~\eqref{eq:nlw0} becomes
\begin{equation}
    \label{eq:nlw}
    \Omega^2\,\partial_{t}^{2}u - \partial_{x}^{2}u + u^{3} = 0\,,
    \quad
    u(t,0) = 0 = u(t,\pi)\,.
\end{equation}
Then we are looking for solutions that are $2\pi$ time-periodic in time $\tau$, i.e., $u(\tau+2\pi,x)=u(\tau,x)$. We focus here on solutions with rational frequencies that can be written as
\begin{equation}
    \label{eq:omega}
    \Omega=\frac{2p+1}{2q}\,,
\end{equation}
where $p,q\in\mathbb{N}_+$. As will become apparent, this choice lets us omit the small divisor problem. Since the considered equation has defocusing nonlinearity, we restrict ourselves only to $\Omega>1$, as this is the case covered in \cite{Ficek.2024A}. Hence, we can just assume $p>q$. Then we have the following result.
\begin{theorem}\label{thm:main}
    For $\Omega=69/40$ and each of $u_0^{(i)}$, $i=1,2,3$ given explicitly in the Supplemental Material, there exists a corresponding solution $u^{(i)}$ to Eq.~\eqref{eq:nlw} close to $u_0^{(i)}$ in the sense $\Vert u^{(i)}-u_0^{(i)} \Vert<\varepsilon^{(i)}$, where $\Vert \cdot \Vert$ is a weighted $\ell^{1}$ norm defined in Section~\ref{sec:FunctionalSpaces}, and $\varepsilon$'s are listed in \eqref{eq:25.05.21_02}. 
    The solutions $\{ \pm u^{(i)} \}_{i=1}^3$ are pairwise distinct, i.e., $u^{(i)} \ne \pm u^{(j)}$ for $i \ne j$.
\end{theorem}
\begin{figure}[t]
    \centering
    \includegraphics[width=0.90\linewidth]{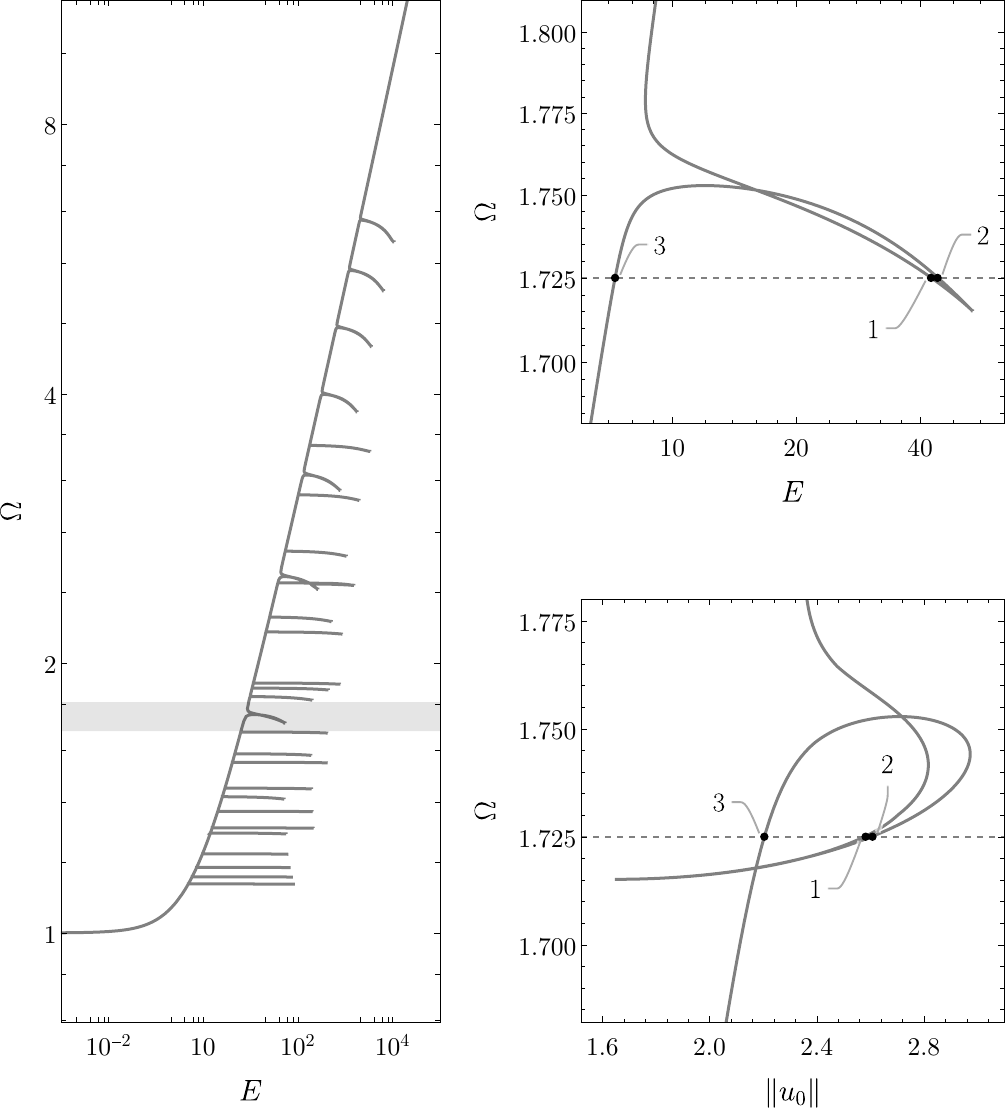}
    \caption{(Left) Frequency diagram of time-periodic solutions computed with a mode truncation $N=M = 9$ (cf.~\cite{Ficek.2024B}), illustrating the global bifurcation structure.
    (Right top) A zoomed-in region around $\Omega=69/40=1.725$, showing the locations of the three approximate solutions $u_0^{(i)}$, \mbox{$i=1,2,3$}, studied in this work. 
    (Right bottom) The same region plotted with respect to the weighted $\ell^{1}$ norm instead of energy $E$, see Eq.~(1.3) in \cite{Ficek.2024A}.}
    \label{fig:EnergyNormFrequency}
\end{figure}
\begin{remark} \normalfont
The value $\Omega=69/40$ has been selected based on Fig.~\ref{fig:EnergyNormFrequency}. On the one hand, our numerical results suggest the existence of interesting solutions (not dominated by the lowest mode) with this frequency.  
Based on these results, we find approximate solutions spanned by sufficiently many modes and then substitute their coefficients with close rational numbers to get $u_0^{(i)}$ listed in the Supplemental Material.
On the other hand, the associated $q$ is not too large, as it would increase the number of necessary calculations. 
\end{remark}

\begin{remark} \normalfont
We are convinced that the same procedure can be applied to other approximate solutions found in \cite{Ficek.2024A}, as long as their frequency $\Omega$ is in the form of \eqref{eq:omega}, resulting in additional examples of the solutions to Eq.~\eqref{eq:nlw0} belonging to this new class. In addition, we expect that by increasing the truncation in the Galerkin method, new approximate solutions with $\Omega=69/40$ emerge, that also can serve as the basis for solutions to Eq.~\eqref{eq:nlw0}. Hence, the results presented in this work are in no regard exhaustive, as typical for computer-assisted methods.
\end{remark}

\begin{remark} \normalfont  
    Recall the $\mathbb{Z}_2$ symmetry ($u\to-u$) of Eq.~\eqref{eq:nlw}. Thus, in order to ensure that the obtained solutions do not coincide, we check that $\{ \pm u^{(i)} \}_{i=1}^3$ are pairwise distinct.
\end{remark}

\begin{remark} \normalfont  
    Thanks to the symmetry between time-periodic solutions of the equations with defocusing and focusing nonlinearities, see Eq.~(2.27) of \cite{Ficek.2024A}, Theorem \ref{thm:main} automatically gives us the existence of time-periodic solutions to the focusing problem with frequency $\Omega=40/96$. 
\end{remark}

This paper can be divided into three parts. After this introduction, in Section~\ref{sec:Strategy} we present the outline of the proof of Theorem \ref{thm:main}. We introduce a proper functional setting and formulate the fixed-point argument leading to the desired result. We show that it is a simple consequence of the technical Theorem \ref{thm:cap}. The second part of the paper, i.e., Sections~\ref{sec:bounds}-\ref{sec:proof}, is devoted to its proof. Finally, the readers interested in the technical details of computer-assisted calculations are directed to the Appendices and attached Supplemental Materials.

\section{Strategy of the proof}\label{sec:Strategy}

\subsection{Functional spaces}\label{sec:FunctionalSpaces}
Let us define a family of \textit{basis functions} $P_{m,n}$
\begin{equation}
    \label{eq:Pmn}
    P_{m,n}(\tau,x):=\cos (2m+1)\tau\,\sin(2n+1)x\,,
\end{equation}
where $m$ and $n$ are non-negative integers. They can be extended to negative indices by defining
\begin{equation}\label{eqn:negative_P}
    P_{-m,n}=P_{m-1,n}\,,\qquad P_{m,-n}=-P_{m,n-1}\,,\qquad P_{-m,-n}=-P_{m-1,n-1}\,.
\end{equation}
Finite linear combinations of functions $P_{m,n}$, i.e., 
\begin{equation}\label{eq:u}
    v(\tau,x)=\sum_{\substack{0\leq m<M\\ 0\leq n< N}}\hat{v}_{m,n}\,P_{m,n}(\tau,x)\,,
\end{equation}
where $M, N\in\mathbb{N}$, constitute the vector space $\tilde{X}$. For any fixed weights $\rhoOne>1$ and $\rhoTwo>1$, we define a norm on $\tilde{X}$ by
\begin{equation}
    \label{eq:norm}
    \Vert v\Vert:=\sum_{\substack{0\leq m< M\\ 0\leq n<N}}\rhoOne^{2m+1}\rhoTwo^{2n+1}|\hat{v}_{m,n}|\,,
\end{equation}
where $v$ is given by \eqref{eq:u}. By $X$ we denote the completion of $\tilde{X}$ with respect to this norm. Then $X$ is a Banach space consisting of $2\pi$-periodic functions in both variables, that satisfy the Dirichlet boundary condition \eqref{eq:nlw}. In addition, for any $v\in X$ the Fourier coefficients decrease as
\begin{equation}
    |\hat{v}_{m,n}|<\Vert v\Vert e^{-(2m+1)\ln\rhoOne} e^{-(2n+1)\ln\rhoTwo}\,.
\end{equation}
As a result, elements of $X$ are analytic inside the strip (cf. \cite{arnold2012geometrical}), 
\begin{equation}
    \left\{(\tau,x)\in\mathbb{C}^2:\Im \tau<\log\rhoOne,\Im x<\log\rhoTwo\right\}\,,
\end{equation}
hence, they are differentiable in the real variables. Finally, as we show in Lemma \ref{lem:triple_norm}, for any $u,v,w\in X$ it holds 
\begin{equation}
    \Vert u\,v\,w\Vert\leq \Vert u\Vert  \Vert v\Vert \Vert w\Vert\,.
\end{equation}
Thus, $X$ is closed with respect to the multiplication of an odd number of its elements, which lets us easily deal with the nonlinearity.

We will denote an open ball of centre $v\in X$ and radius $r>0$ with respect to the norm $\Vert\cdot\Vert$ as $B_r(v)$. For two non-negative integers $M$ and $N$, we define the following subspaces of $X$:
\begin{align}
    \label{eq:XMN}
    X_{M,N}:=\left\{v\in X: \hat{v}_{m,n}=0\mbox{ for } m<M \mbox{ or } n<N\right\}\,,\\
    \Yfin_{M,N}:=\left\{v\in X: \hat{v}_{m,n}=0\mbox{ for } m\geq M \mbox{ or } n\geq N\right\}\,.
\end{align}
Then we have the decomposition $X= \Yfin_{M,N}\oplus{}(X_{M,0}+X_{0,N})$. Note that it is not a direct sum since $X_{M,0}$ and $X_{0,N}$ have a non-trivial intersection, see Fig.~\ref{fig:XMN}. The space of continuous linear operators on $X$ will be denoted by $\mathcal{L}(X)$.

\begin{figure}[t]
    \centering
    \includegraphics[width=0.5\linewidth]{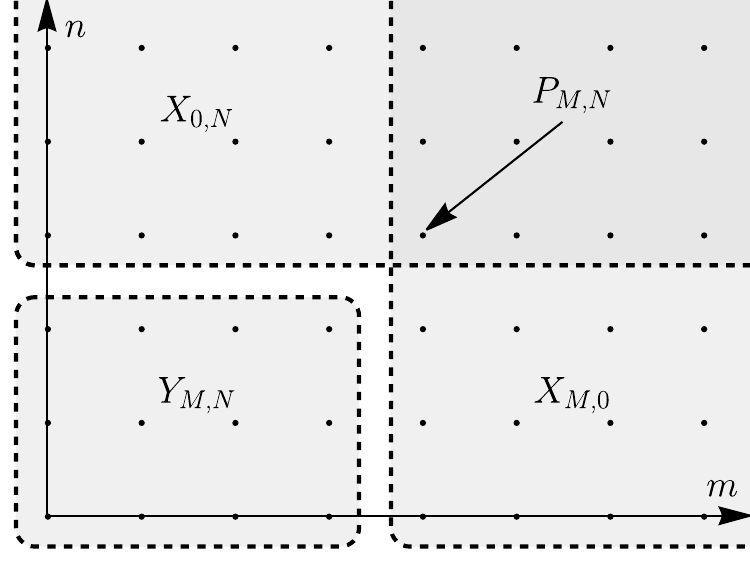}
    \caption{Diagram presenting spaces $\Yfin_{M,N}$, $X_{M,0}$, and $X_{0,N}$ for some fixed $M$ and $N$. Dots represent basis functions $P_{m,n}$ belonging to appropriate subspaces. The coordinate axes intersect at the point $(0,0)$.}
    \label{fig:XMN}
\end{figure}

\subsection{Fixed-point formulation}
\label{sec:FixedPointFormulation}
Let us denote by $L_\Omega$ a linear operator
\begin{equation}
    \label{eq:L}
    L_\Omega=\Omega^2\,\partial_\tau^2-\partial_x^2\,.
\end{equation}
Then we can rewrite \eqref{eq:nlw} as a fixed-point problem in space $X$
\begin{equation}
    \label{eq:fp}
    u=-L_\Omega^{-1}u^3=:\mathcal{F}_\Omega(u)\,.
\end{equation}
As Lemma \ref{lem:L-1} shows, $L_\Omega^{-1}$ is a well-defined, bounded linear operator in $X$. We can represent the sought solution as $u=u_0+Ah$, where $u_0\in X$ will be the approximation to the exact solution, $A:X\to X$ is a linear, continuous isomorphism to be fixed later, and $h\in X$ is the new unknown. Then we define a nonlinear operator $\mathcal{N}_\Omega$ acting on $X$
\begin{equation}
    \label{eq:fp_2}
    \mathcal{N}_\Omega(h):=\mathcal{F}_\Omega(u_0+Ah)-u_0+(I-A)h\,.
\end{equation}
Clearly $h$ is a fixed point of $\mathcal{N}_\Omega$ if and only if $u=u_0+Ah$ is a fixed point of $\mathcal{F}_\Omega$. Let us introduce a linear $H_0:X\to X$
\begin{align}\label{eq:H0}
    H_0(h)=-3 L_\Omega^{-1}\left(u_0^2 \,Ah\right)+h-Ah\,,
\end{align}
then for any $h_1, h_2\in X$ we have
\begin{align}
    \mathcal{N}_\Omega(h_1)-\mathcal{N}_\Omega(h_2)=&H_0(h_1-h_2)-3L_\Omega^{-1}\left(u_0\,(Ah_1+Ah_2)(Ah_1-A h_2)\right)\\
    &-L_\Omega^{-1}\left[\left((Ah_1)^2+Ah_1\,Ah_2+(Ah_2)^2\right)(Ah_1-A h_2)\right]\,.
\end{align}
If we restrict to $h_1,h_2\in B_\delta(0)$, then it holds
\begin{align}\label{eq:NN_norm}
    \left\Vert\mathcal{N}_\Omega(h_1)-\mathcal{N}_\Omega(h_2)\right\Vert\leq
    \left(\Vert H_0\Vert+6\left\Vert L_\Omega^{-1}\right\Vert \Vert u_0\Vert \Vert A \Vert^2 \delta+3\left\Vert L_\Omega^{-1}\right\Vert \Vert A \Vert^3 \delta^2 \right)\Vert h_1-h_2\Vert\,.
\end{align}

\subsection{Proof of Theorem 1}

Let us now assume that the following result holds.

\begin{theorem}\label{thm:cap}
    For $\Omega=69/40$ and each of $u_0\in\{u_0^{(1)},u_0^{(2)},u_0^{(3)}\}$, provided explicitly in the Supplemental Material, there exists a respective linear isomorphism $A:X\to X$ and resulting $H_0$, given by Eq.~\eqref{eq:H0}, such that one can find positive constants $\delta$ and $K_0$ satisfying
\begin{equation}
    \label{eq:25.05.21_05}
    \Vert H_0\Vert+6\left\Vert L_\Omega^{-1}\right\Vert \Vert u_0\Vert \Vert A \Vert^2 \delta+3\left\Vert L_\Omega^{-1}\right\Vert \Vert A \Vert^3 \delta^2 <K_0<1\,,\qquad
    \Vert \mathcal{N}_\Omega(0)\Vert <(1-K_0)\delta
    \,.
\end{equation}
In addition, if we introduce $\varepsilon:=\Vert A\Vert\delta$, then the values $\varepsilon^{(i)}$ for respective $u_0^{(i)}$ satisfy
\begin{align}\label{eq:u-u0}
   \left\Vert u_0^{(i)} \pm u_0^{(j)}\right\Vert >\varepsilon^{(i)}+\varepsilon^{(j)}\,
\end{align}
for $i,j\in\{1,2,3\}$ and $i\neq j$.
\end{theorem}

Then the Theorem~\ref{thm:main} immediately follows from Theorem~\ref{thm:cap}.

\begin{proof}[Proof of Theorem \ref{thm:main}]
Theorem \ref{thm:cap} implies that $\mathcal{N}_\Omega$ can be considered as a map \mbox{$B_{\delta}(0)\to B_\delta (0)$}, since
\begin{align}
\label{eq:25.05.21_04}
    \left\Vert\mathcal{N}_\Omega(h)\right\Vert\leq\left\Vert\mathcal{N}_\Omega(h)-\mathcal{N}_\Omega(0)\right\Vert+ \left\Vert\mathcal{N}_\Omega(0)\right\Vert < K_0 \delta+(1-K_0)\delta=\delta\,.
\end{align}
In addition, from Eq.~\eqref{eq:NN_norm} we infer that $\mathcal{N}_\Omega$ is a contraction inside $B_{\delta}(0)$. As a result, we can employ the Banach contraction principle to find the fixed point $h$ of $\mathcal{N}_\Omega$. It leads to the solution $u$ \eqref{eq:nlw} expressed as $u=u_0+Ah$. Since $h\in B_\delta(0)$, the distance between the solution $u$ and its approximation $u_0$ is bounded by
\begin{align}\label{eq:epsilon}
   \Vert u-u_0\Vert\leq\Vert A h\Vert\leq \Vert A \Vert \delta = \varepsilon\,.
\end{align}
Then the fact that the solutions $\{\pm u^{(i)}\}_{i=1}^3$ are distinct follows from \eqref{eq:u-u0}, since for $i,j\in\{1,2,3\}$ and $i\neq j$ we have
\begin{align}
   \left\Vert u^{(i)}\pm u^{(j)}\right\Vert\geq \left\Vert u_0^{(i)}\pm u_0^{(j)}\right\Vert-\left\Vert u^{(i)}-u_0^{(i)}\right\Vert-\left\Vert u^{(j)}-u_0^{(j)}\right\Vert>0\,.
\end{align}
\end{proof}

The remaining part of this paper is devoted to the proof of Theorem \ref{thm:cap}. In Section~\ref{sec:bounds} we state four lemmas letting us control the norms of the considered operators. In order for our fixed-point argument to close, we need to properly construct an auxiliary operator $A$. This matter is discussed in Section~\ref{sec:A}. Results of these two sections lead us to expressions for the bounds on the norm of operator $H_0$, derived in Section~\ref{sec:H0}. These bounds are then calculated explicitly with the assistance of a computer using rational arithmetic, as described in Section~\ref{sec:ComputerAssistance}. Finally, in Section~\ref{sec:proof} we combine the obtained bounds to show the existence of suitable values of $K_0$ and $\delta$ and we verify \eqref{eq:u-u0}, thus concluding the proof.

\section{Bounds on operator norms}\label{sec:bounds}
We begin this section, devoted to deriving bounds on specific operator norms, by establishing an estimate for the triple product of elements in $X$.
\begin{lemma}\label{lem:triple_norm}
    Let $u,v,w\in X$, then
    \begin{align}
        \Vert u\, v\, w\Vert\leq \Vert u\Vert\,\Vert v\Vert\,\Vert w \Vert\,.
    \end{align}
\end{lemma}
\begin{proof}
    The decomposition of a triple product of $P_{m,n}$ leads to 
    \begin{align}
        P_{m_1,n_1}\,P_{m_2,n_2}\,P_{m_3,n_3}=&\frac{1}{16}\left[-P_{m_1+m_2+m_3+1,n_1+n_2+n_3+1}-P_{-m_1+m_2+m_3,n_1+n_2+n_3+1}\right.\\
        &-P_{m_1-m_2+m_3,n_1+n_2+n_3+1}- P_{m_1+m_2-m_3,n_1+n_2+n_3+1}\\
        &+P_{m_1+m_2+m_3+1,-n_1+n_2+n_3}+P_{-m_1+m_2+m_3,-n_1+n_2+n_3}\\
        &+P_{m_1-m_2+m_3,-n_1+n_2+n_3}+ P_{m_1+m_2-m_3,-n_1+n_2+n_3}\\
        &+P_{m_1+m_2+m_3+1,n_1-n_2+n_3}+P_{-m_1+m_2+m_3,n_1-n_2+n_3}\\
        &+P_{m_1-m_2+m_3,n_1-n_2+n_3}+ P_{m_1+m_2-m_3,n_1-n_2+n_3}\\
        &+P_{m_1+m_2+m_3+1,n_1+n_2-n_3}+P_{-m_1+m_2+m_3,n_1+n_2-n_3}\\
        &\left.+P_{m_1-m_2+m_3,n_1+n_2-n_3}+ P_{m_1+m_2-m_3,n_1+n_2-n_3}\right]\,,
    \end{align}
where potential negative indices are covered by \eqref{eqn:negative_P}. Thus, it holds
\begin{align}
        \left\Vert P_{m_1,n_1}\,P_{m_2,n_2}\,P_{m_3,n_3}\right\Vert \leq \rhoOne^{2m_1+2m_2+2m_3+3}\rhoTwo^{2n_1+2n_2+2n_3+3}\,.
\end{align} 

Now assume that $u,v,w\in \tilde{X}$, so we have finite decompositions
    \begin{align}
        u=\sum_{\substack{0\leq m< M_1\\ 0\leq n<N_1}} \hat{u}_{m,n}P_{m,n}\,,\qquad 
        v=\sum_{\substack{0\leq m< M_2\\ 0\leq n<N_2}} \hat{v}_{m,n}P_{m,n}\,,\quad 
        w=\sum_{\substack{0\leq m< M_3\\ 0\leq n<N_3}} \hat{w}_{m,n}P_{m,n}\,.
    \end{align}
Then
    \begin{align}
        \Vert u\,v\,w\Vert&\leq 
        \sum_{\substack{0\leq m_1< M_1\\ 0\leq n_1<N_1}}
        \sum_{\substack{0\leq m_2< M_2\\ 0\leq n_2<N_2}} 
        \sum_{\substack{0\leq m_3< M_3\\ 0\leq n_3<N_3}} 
        \left|\hat{u}_{m_1,n_1} \hat{v}_{m_2,n_2} \hat{w}_{m_3,n_3}\right| \left\Vert P_{m_1,n_1}P_{m_2,n_2} P_{m_3,n_3}\right\Vert\\
        &\leq
        \sum_{\substack{0\leq m_1 < M_1\\ 0\leq n_1 < N_1}} \sum_{\substack{0\leq m_2 < M_2\\ 0\leq n_2 < N_2}} \sum_{\substack{0\leq m_3 < M_3\\ 0\leq n_3 < N_3}} \left|\hat{u}_{m_1,n_1} \hat{v}_{m_2,n_2} \hat{w}_{m_3,n_3}\right|\\
        &\qquad\qquad\qquad\qquad\qquad\qquad\qquad\qquad \times \rhoOne^{2m_1+2m_2+2m_3+3}\rhoTwo^{2n_1+2n_2+2n_3+3} \\
        &=\left(\sum_{\substack{0\leq m_1< M_1\\ 0\leq n_1 < N_1}}\rhoOne^{2m_1+1}\rhoTwo^{2n_1+1} \left|\hat{u}_{m_1,n_1}\right| \right) \left(\sum_{\substack{0\leq m_2 < M_2\\ 0\leq n_2 < N_2}} \rhoOne^{2m_2+1}\rhoTwo^{2n_2+1}\left|\hat{v}_{m_2,n_2}\right| \right) \\
        &\quad \times\left(\sum_{\substack{0\leq m_3 < M_3\\ 0\leq n_3< N_3}} \rhoOne^{2m_3+1}\rhoTwo^{2n_3+1}\left|\hat{u}_{m_1,n_1}\right| \right)
        =\Vert u\Vert\,\Vert v\Vert\,\Vert w \Vert\,.
    \end{align}
    As $X$ is the completion of $\tilde{X}$, this concludes the proof.
\end{proof}

Since $\Vert\cdot\Vert$ is a weighted $\ell^1$ norm, any continuous linear operator has a norm that can be bounded from above as follows.
\begin{lemma}\label{lem:lin_norm}
Let $\mathcal{H}\in \mathcal{L}(X)$, then
\begin{align}\label{eq:lin_norm}
    \Vert \mathcal{H}\Vert\leq \sup_{m,n}\frac{\left\Vert \mathcal{H}\left(P_{m,n}\right)\right\Vert}{\rhoOne^{2m+1} \rhoTwo^{2n+1}} \,.
\end{align}    
\end{lemma}
\begin{proof}
Any $h\in X$ can be represented as
\begin{align}
   h=\sum_{m,n}\hat{h}_{m,n}\, P_{m,n}\,.
\end{align}    
Then we have
\begin{align}
\left\Vert\mathcal{H}\left(h\right)\right\Vert& \leq \sum_{m,n}\left|\hat{h}_{m,n}\right|\left\Vert \mathcal{H}\left(P_{m,n}\right)\right\Vert = \sum_{m,n}\left( \rhoOne^{2m+1}\rhoTwo^{2n+1}\left|\hat{h}_{m,n}\right|\right)\frac{\left\Vert \mathcal{H}\left(P_{m,n}\right) \right\Vert}{\rhoOne^{2m+1} \rhoTwo^{2n+1}}\\
&\leq \left(\sup_{m,n}\frac{\left\Vert \mathcal{H}\left(P_{m_1,n_1}\right)\right\Vert}{\rhoOne^{2m+1} \rhoTwo^{2n+1}}\right)\sum_{m,n} \rhoOne^{2m+1} \rhoTwo^{2n+1}\left|\hat{h}_{m,n}\right|= \left(\sup_{m,n}\frac{\left\Vert \mathcal{H}\left(P_{m,n}\right)\right\Vert}{\rhoOne^{2m+1} \rhoTwo^{2n+1}}\right)\left\Vert h\right\Vert\,.
\end{align}    
\end{proof}

Lemma \ref{lem:lin_norm} is the main tool in finding bounds on norms of operators $A$ and $H_0$ in Sections~\ref{sec:A} and \ref{sec:H0}, respectively. In this process, we need proper control over the $L_\Omega^{-1}$. It is ensured by the following result.

\begin{lemma}\label{lem:L-1}
    The operator $L_\Omega:X\to X$ has a continuous inverse. In addition, for any $v\in X_{m,n}$ it holds $\left\Vert L^{-1}_\Omega  v\right\Vert\leq \phi(m,n)\Vert v\Vert$, where
    \begin{equation}
        \phi(m,n)=\frac{4q^2}{2\max\left(2q(2n+1),(2p+1)(2m+1)\right)-1}\,.
    \end{equation}
\end{lemma}
\begin{proof}
    Let $t$ and $s$ be positive numbers such that $|t-s|\geq 1$. Then the following inequality holds
\begin{equation}
\label{eq:ineq}
    |t^2-s^2|\geq 2\max(t,s)-1\,.
\end{equation}
To prove it, assume that $s<t$ and consider a function $f(s)=t^2-s^2$ defined on an interval $s\in(0,t-1]$ for some fixed $t$. The function $f$ achieves its minimum at $s=t-1$, hence $(t^2-s^2)\geq 2t-1$. The case when $s>t$ is analogous, leading to \eqref{eq:ineq}.

Now, for any $P_{m,n}$ we have
\begin{equation}
L^{-1}_\Omega P_{m,n}=\frac{1}{(2n+1)^2-\Omega^2 (2m+1)^2}P_{m,n}=\frac{4q^2}{4q^2(2n+1)^2-(2p+1)^2(2m+1)^2}P_{m,n}\,.
\end{equation}
Since the denominator of the fraction is a difference of an even and odd number, it is at least equal to one and $L^{-1}_\Omega P_{m,n}$ is well defined. In addition, by setting $t=2q\,(2n+1)$, $s=(2p+1)(2m+1)$ in \eqref{eq:ineq}, we get
\begin{equation}
\left|\frac{4q^2}{4q^2(2n+1)^2-(2p+1)^2(2m+1)^2}\right|\leq
\phi(m,n)\,.
\end{equation}
Thus, we obtain the upper bound
\begin{equation}
\left\Vert L^{-1}_\Omega P_{m,n}\right\Vert\leq \phi(m,n)\left\Vert P_{m,n}\right\Vert\,.
\end{equation}
The function $\phi$ is nonincreasing in $m$ and $n$. As a result, for any $v \in X_{m,n}$ it holds
\begin{equation}
\left\Vert L^{-1}_\Omega v\right\Vert\leq \phi(m,n)\Vert v\Vert\,.
\end{equation}
In particular, since $X_{0,0}=X$, we have
\begin{equation}
\left\Vert L^{-1}_\Omega\right\Vert\leq \phi(0,0)=\frac{4q^2}{4p+1}\,.
\end{equation}
\end{proof}

The final lemma gives us control over the way in which the first part of operator $H_0$, see Eq.~\eqref{eq:H0}, acts on higher modes.

\begin{lemma}\label{lem:uvP_norm}
    Let $v\in X$ such that
    \begin{equation}\label{eq:25.06.03_01}
    v=\sum_{\substack{0\leq m< M\\ 0\leq n<N}}\hat{v}_{m,n}P_{m,n}\,.
    \end{equation}
    Assume that $m\geq \hat{M}$ or $n\geq \hat{N}$, where
    \begin{equation}\label{eq:25.06.03_02}
    \hat{M}= 2M-1 \,,\qquad \hat{N}= 2N-1 \,.
    \end{equation} 
    Then it holds
    \begin{equation}
    \left\Vert L_\Omega^{-1}\left(v^2 P_{m,n}\right) \right\Vert 
    \leq \max\left(\phi(m-\hat{M},0),\phi(0,n-\hat{N})\right) \rhoOne^{2(\hat{M}+m)+1}\rhoTwo^{2(\hat{N}+n)+1}\sum_{\hat{m}=0}^{2\hat{M}}\sum_{\hat{n}=0}^{2\hat{N}}\left| \tilde{c}_{\hat{m},\hat{n}} \right|\,,
    \end{equation}
    where the coefficients $\tilde{c}_{\hat{m},\hat{n}}$ are defined in \eqref{eq:uvP_decompositionMN} and depend only on $v$.
\end{lemma}
\begin{proof}
    For $v$ given in \eqref{eq:25.06.03_01} and $P_{\hat{M},\hat{N}}$ with \eqref{eq:25.06.03_02}, we have the following decomposition
    \begin{align}\label{eq:uvP_decompositionMN}
        v^2P_{\hat{M},\hat{N}}= \sum_{\hat{m}=0}^{2\hat{M}}\sum_{\hat{n}=0}^{2\hat{N}}\tilde{c}_{\hat{m},\hat{n}}P_{\hat{m},\hat{n}}\,,
    \end{align}
    where the coefficients $\tilde{c}_{\hat{m},\hat{n}}$ depend only on $v$. For any other pair of natural numbers $(m,n)$ we can write
    \begin{align}\label{eq:uvP_decomposition}
        v^2P_{m,n}=\sum_{\hat{m}=0}^{2\hat{M}}\sum_{\hat{n}=0}^{2\hat{N}}\tilde{c}_{\hat{m},\hat{n}} P_{\hat{m}+m-\hat{M},\hat{n}+n-\hat{N}}\,.
    \end{align}
    Let us point out that when $m<\hat{M}$ or $n<\hat{N}$, then this sum contains basis functions with negative indices. They are treated accordingly to \eqref{eqn:negative_P}. 
    
    Since $\Vert P_{m,n}\Vert=\rhoOne^{2m+1}\rhoTwo^{2n+1}$ we also have the bound
    \begin{align}
        \left\Vert v^{2}\,P_{m,n}\right\Vert&\leq\sum_{\hat{m}=0}^{2\hat{M}}\sum_{\hat{n}=0}^{2\hat{N}}\rhoOne^{2(\hat{m}+m-\hat{M})+1}\rhoTwo^{2(\hat{n}+n-\hat{N})+1} \left|\tilde{c}_{\hat{m},\hat{n}}\right|\\
        &\leq \rhoOne^{2(\hat{M}+m)+1}\rhoTwo^{2(\hat{N}+n)+1}  \sum_{\hat{m}=0}^{2\hat{M}}\sum_{\hat{n}=0}^{2\hat{N}}\left|\tilde{c}_{\hat{m},\hat{n}}\right| \,.
    \end{align}
    Now assume that $m\geq \tilde{M}$, then $v^2P_{m,n}\in X_{m-\hat{M},0}$ and by Lemma \ref{lem:L-1} we have
    \begin{align}
    \label{eq:25.05.20_01}
     \left\Vert L^{-1}_{\Omega}\left(v^2 P_{m,n}\right)\right\Vert&\leq
     \phi(m-\hat{M},0) \rhoOne^{2(\hat{M}+m)+1}\rhoTwo^{2(\hat{N}+n)+1} \sum_{\hat{m}=0}^{2\hat{M}}\sum_{\hat{n}=0}^{2\hat{N}} \left| \tilde{c}_{\hat{m},\hat{n}} \right|\,.
    \end{align}    
    Analogously, for $n\geq \hat{N}$, it holds
    \begin{align}
    \label{eq:25.05.20_02}
     \left\Vert L^{-1}_{\Omega}\left(v^2 P_{m,n}\right)\right\Vert&\leq
     \phi(0,n-\hat{N}) \rhoOne^{2(\hat{M}+m)+1}\rhoTwo^{2(\hat{N}+n)+1} \sum_{\hat{m}=0}^{2\hat{M}}\sum_{\hat{n}=0}^{2\hat{N}} \left| \tilde{c}_{\hat{m},\hat{n}} \right|\,.
    \end{align}        
\end{proof}

\section{Construction of $A$}\label{sec:A}
In order for the argument presented in Section~\ref{sec:Strategy} to close, the norm of the operator $H_0$ needs to be smaller than one. It can be achieved by a proper choice of the operator $A$. In Supplemental Material, we provide explicit expressions for this operator that are used in the proof. The goal of this section is to present a reasoning that brought us to them.

For any fixed $u_0=u_0^{(i)}$ we construct the operator $A$ as a finite-dimensional approximation of the inverse of operator $I+3L_\Omega^{-1}\circ\Lambda_{u_0^2}$, where $\Lambda_{u_0^2}$ is a pointwise multiplication, i.e., $\Lambda_{u_0^2}h=u_0^2h$. One can easily check that for $A=(I+3L_\Omega^{-1}\circ\Lambda_{u_0^2})^{-1}$ we would get $H_0=0$. Unfortunately, getting an explicit expression for such $A$ does not seem to be possible. Thus, we construct an approximation as follows.

For fixed natural numbers $\mu$ and $\nu$, we denote the projection onto $\Yfin_{\mu,\nu}$ by $\Pi_{\mu,\nu}$:
\begin{equation}
    \Pi_{\mu,\nu}u=\sum_{\substack{0\leq m < \mu\\ 0\leq n < \nu}}u_{m,n}P_{m,n}\,.
\end{equation}
Let us define an auxiliary linear operator $\tilde{\mathcal{A}}:\Yfin_{\mu,\nu}\to Y_{\mu,\nu}$ as
\begin{equation}
    \tilde{\mathcal{A}}:=I+3\Pi_{\mu,\nu}\circ L_\Omega^{-1}\circ \Lambda_{u_0^2}\,.
\end{equation}
It can be presented in the basis consisting of $P_{m,n}$ with $0\leq m <\mu$ and $0\leq n < \nu$ as a matrix of dimension $\mu\,\nu\times\mu\nu$. Then one could, in principle, explicitly calculate its inverse and define $\mathcal{A}=\tilde{\mathcal{A}}^{-1}$. However, as discussed in Section~\ref{sec:ComputerAssistance}, for fixed $\mu$ and $\nu$, $\tilde{\mathcal{A}}$ is constructed as a matrix with entries being rational numbers. Because of this, calculating its inverse may be bothersome and produce $\mathcal{A}$ filled with rational numbers with unfeasibly large numerators and denominators, complicating further computations. To circumnavigate this issue, we approximate matrix elements of $\tilde{\mathcal{A}}$ as floating-point numbers, find an approximate inverse of this matrix using standard libraries, and use rational approximations of the obtained entries to define $\mathcal{A}$. This procedure gives us a matrix (listed in Supplemental Materials) that is not an exact inverse of $\tilde{\mathcal{A}}$, but is sufficiently close to it for our needs.

Having proper matrix $\mathcal{A}$, finally we define $A$ as a block-diagonal operator
\begin{equation}
\label{eq:25.05.21_01}
A=\left(\begin{array}{@{}c|c@{}}
  \mathcal{A}
  & 0 \\
\hline
  0 & I
\end{array}\right)\,.
\end{equation}
Let us note that $\mu$ and $\nu$ are free parameters. One can expect that the larger they are, the better approximation of $(I+3L_\Omega^{-1}\circ\Lambda_{u_0^2})^{-1}$ the operator $A$ is. This can be measured by the size of the resulting $\Vert H_0\Vert$. For our proof, we choose $\mu=\nu$ as a minimal value for which $\Vert H_0\Vert<1$. 

Due to the block structure, we have $A P_{m,n}=\mathcal{A}P_{m,n}$ for $P_{m,n}\in \Yfin_{\mu,\nu}$, while $A P_{m,n}=P_{m,n}$ when $P_{m,n}\in X_{\mu,0}$ or $P_{m,n}\in X_{0,\nu}$. The norm of operator $A$ can then be bounded with the use of Lemma \ref{lem:lin_norm} as
\begin{equation}
\label{eq:25.05.21_03}
\Vert A \Vert \leq \max\left(\max_{\substack{0\leq m<\mu\\ 0\leq n<\nu}} \frac{\left\Vert \mathcal{A}P_{m,n}\right\Vert}{\rhoOne^{2m+1}\rhoTwo^{2n+1}},1\right)\,.
\end{equation}

\section{Bound on $\Vert H_{0}\Vert$}\label{sec:H0}
Let $u_0$ be given by a finite trigonometric polynomial
\begin{equation}
    \label{eq:25.05.22_01}
    u_0=\sum_{\substack{0\leq m< M\\ 0\leq n< N}}c_{m,n} P_{m,n}
\end{equation}
and $A$ be constructed as described in the previous section. For any fixed $\tilde{M}\geq \max(\mu, 2M-1)$ and $\tilde{N}\geq \max(\nu,2N-1)$, Lemma \ref{lem:lin_norm} gives us the bound
\begin{equation}\label{eq:H0_bound_proof}
    \Vert H_0\Vert\leq \max\left(\max_{\substack{0\leq m< \tilde{M}\\0\leq n<\tilde{N}}} \frac{\left\Vert H_0 (P_{m,n})\right\Vert}{\rhoOne^{2m+1}\rhoTwo^{2n+1}},\sup_{m\geq \tilde{M} \,\lor\, n\geq\tilde{N}} \frac{\left\Vert H_0 (P_{m,n})\right\Vert}{\rhoOne^{2m+1}\rhoTwo^{2n+1}}\right)\,.
\end{equation}
The first part inside the maximum can be calculated explicitly; hence, we only need to find bounds on the remaining supremum. However, for indices included there, it holds $AP_{m,n}=P_{m,n}$, so we simply have
\begin{equation}
    H_0 (P_{m,n})=-3L^{-1}_\Omega (u_0^2\,P_{m,n})\,.
\end{equation}
As a result, we can use Lemma \ref{lem:uvP_norm} to get 
\begin{align}\label{eq:25.05.20_03}
    \Vert H_0\Vert=\max\left(\max_{\substack{0\leq m< \tilde{M}\\0\leq n<\tilde{N}}} \frac{\left\Vert H_0 (P_{m,n})\right\Vert}{\rhoOne^{2m+1}\rhoTwo^{2n+1}}, \phi(\tilde{M}-(2M-1),0) C,\phi(0,\tilde{N}-(2N-1)) C\right)\,,
\end{align}
where
\begin{align}\label{eq:25.05.20_04}
   C=3\rhoOne^{4M-2}\rhoTwo^{4N-2} \sum_{\tilde{m}=0}^{4M-2}\sum_{\tilde{n}=0}^{4N-2}\left| \tilde{c}_{\tilde{m},\tilde{n}} \right|\,.
\end{align}
The coefficients $\tilde{c}_{\tilde{m},\tilde{n}}$ can be read from the decomposition
\begin{align}
    \label{eq:25.06.03_03}
    u_0^2\,P_{2M-1,2N-1}= \sum_{\tilde{m}=0}^{4M-2}\sum_{\tilde{n}=0}^{4N-2}\tilde{c}_{\tilde{m},\tilde{n}}P_{\tilde{m},\tilde{n}}\,.
\end{align}
see \eqref{eq:uvP_decompositionMN}.

\section{Computer assistance}
\label{sec:ComputerAssistance}

This section outlines the computational aspects of the proof, which rigorously verify the hypotheses of Theorem~\ref{thm:cap}, using exclusively rational arithmetic.
Given the rational Fourier coefficients of the approximate solution $u_0$---truncated to a finite number of temporal and spatial modes \eqref{eq:25.05.22_01}---and the frequency $\Omega$ of the form given in~\eqref{eq:omega}, the objective is to confirm the conditions \eqref{eq:25.05.21_05} that guarantee the existence of a true solution in a neighbourhood of $u_0$, as quantified by the weighted $\ell^1$ norm.

The procedure begins by evaluating the norm $\left\Vert u_0\right\Vert$ and the norm bound on the linear operator A, as defined in \eqref{eq:25.05.21_03}, for fixed weights $\rhoOne$ and $\rhoTwo$. We then compute the Fourier coefficients of $u_0^2$ and $u_0^3$ using trigonometric identities and convolution formulas for Fourier series. These are used to construct $\mathcal{N}_\Omega(0)$, defined in \eqref{eq:fp_2}, and to evaluate its norm.

Next, we compute the norm bound on the operator $H_0$ using~\eqref{eq:25.05.20_03}. It involves a max expression, where the first term concerns $P_{m,n}$ in $\Yfin_{\tilde{M},\tilde{N}}$ subspace. This can be split into two cases: the first with arguments in $\Yfin_{\mu,\nu}$ for which $A = \mathcal{A}$ and the second for the remaining elements of $\Yfin_{\tilde{M},\tilde{N}}$ where $A = I$. These cases are computed separately; the latter requires checking a larger number of expressions and therefore accounts for the majority of the computational effort. The remaining terms in \eqref{eq:25.05.20_03} are computed straightforwardly.

The choice of truncation parameters $M$, $N$, $\mu$, $\nu$, $\tilde{M}$, and $\tilde{N}$ is governed by the need to ensure all bounds fall within the thresholds required by the Theorem~\ref{thm:cap}. The number of Fourier modes in the approximation $u_0$, $M$ temporal and $N$ spatial, is selected so that the norm of $\mathcal{N}_\Omega(0)$ is sufficiently small; in other words, $u_{0}$ is sufficiently close to the true solution. The size of the matrix block $\mathcal{A}\in\mathbb{R}^{\mu\nu\times \mu\nu}$, defined in Section~\ref{sec:A}, is chosen individually for each solution to ensure that the first term in \eqref{eq:25.05.20_03} remains strictly below 1. The parameters $\tilde{M}$ and $\tilde{N}$ are chosen analogously, to bound the remaining two terms in~\eqref{eq:25.05.20_03}.

Finally, we verify whether the inequalities \eqref{eq:25.05.21_05} are satisfied for the provided values of $K_0$ and $\delta$ (which are computed based on the already available data). If these hold, the existence of a true solution near each approximate $u_0^{(i)}$ is thereby established. A separate verification is then performed to confirm that these solutions are distinct from one another, completing the proof.

More explicit formulas, including precise definitions of all computed quantities and norm expressions, are provided in Appendix~\ref{sec:DetailsOfTheComputerCode}.

\section{Proof of Theorem 2}\label{sec:proof}

\begin{proof}[Proof of Theorem~\ref{thm:cap}]
We consider the case $\Omega = 69/40$ and the three approximate solutions $\{ u_0^{(i)} \}_{i=1}^3$ with the corresponding matrices $\mathcal{A}^{(i)}$ listed in the Supplemental Material. For each $i$, the hypotheses of Theorem~\ref{thm:cap} are verified using rigorously computed bounds on the relevant quantities. All computations are performed in exact rational arithmetic and yield the following results:

\begin{itemize}
\item 1st solution: $u_{0}=u_{0}^{(1)}$
\begin{align}
\left\Vert u_{0}\right\Vert &< \frac{14771161379}{5719729173} \,,
\\
\left\Vert A \right\Vert & < \frac{56993842159}{3536278048} \,,
\\
\left\Vert H_0 \right\Vert & < \frac{10426318695}{10439379622} \,,
\\
K_{0} &= \frac{6488907641}{6496696266} \,,
\\
\delta &= \frac{530080}{476772484993687} \,,
\end{align}

\item 2nd solution: $u_{0}=u_{0}^{(2)}$
\begin{align}
\left\Vert u_{0}\right\Vert &< \frac{49511716195}{18983051418} \,,
\\
\left\Vert A \right\Vert & < \frac{31937835596}{2007878747} \,,
\\
\left\Vert H_0 \right\Vert & < \frac{6036122721}{6038803058} \,,
\\
K_{0} &= \frac{5478308396}{5480517807} \,,
\\
\delta &= \frac{110731}{125765612890970} \,,
\end{align}

\item 3rd solution: $u_{0}=u_{0}^{(3)}$
\begin{align}
\left\Vert u_{0}\right\Vert &< \frac{8666442879}{3931226470} \,,
\\
\left\Vert A \right\Vert & < \frac{41051476037}{3576023091} \,,
\\
\left\Vert H_0 \right\Vert & < \frac{5350490449}{5358606877} \,,
\\
K_{0} &= \frac{3117063509}{3120585438} \,,
\\
\delta &= \frac{142842}{7532565418067} \,.
\end{align}
\end{itemize}

Given these values, we apply the fixed-point argument described in Section~\ref{sec:FixedPointFormulation} and conclude that, for each $i = 1, 2, 3$, there exists a true solution $u^{(i)}$ to Eq.~\eqref{eq:nlw} such that
$\| u^{(i)} - u_0^{(i)} \| < \varepsilon^{(i)}$, where
\begin{equation}
\label{eq:25.05.21_02}
\varepsilon^{(1)}= 1.79189\ldots\,\times 10^{-8} \,,
\quad
\varepsilon^{(2)}= 1.40047\ldots\,\times 10^{-8} \,,
\quad
\varepsilon^{(3)}=2.17692\ldots\,\times 10^{-7} \,.
\end{equation}
These latter bounds are computed based on the quantities listed above.

Finally, we verify that inequalities \eqref{eq:u-u0} hold for all pairs of distinct $i,j\in\{1,2,3\}$. It is done by simply calculating $\| u^{(i)} \pm u^{(j)} \|$ and using the values of $\varepsilon$ provided above.

\end{proof}

\appendix

\vspace{8ex}

\section{Description of auxiliary files}
\label{sec:InteractionCoefficients}

To support verification and reproducibility of the computer-assisted proof, we provide a complete set of supplementary files, consisting of \textit{Wolfram Mathematica} \cite{Mathematica} scripts and precomputed data files. These include implementation code, input data, and scripts configured to verify the operator norm bounds and the inequality appearing in the Theorem~\ref{thm:cap}.

\subsection{Code}

The following \textit{Wolfram Mathematica} files are included as the Supplemental Material:
\begin{itemize}
\item \codett{CAPCode.m}: Contains auxiliary functions used in the implementation of the proof procedure. The norm weight parameters \codett{RhoTau} and \codett{RhoX} are defined there.
\item \codett{CAPRun_69_40_1st.m}: A run script tailored to the first solution. It loads \codett{CAPCode.m}, sets all required parameters, reads the corresponding data files, and verifies the bounds and main inequality \eqref{eq:25.05.21_05}. This script constitutes the execution of the proof.
\item \codett{CAPRun_69_40_2nd.m}: As above, for the second solution.
\item \codett{CAPRun_69_40_3rd.m}: As above, for the third solution.
\item \codett{CAPRun_69_40_dist.m}: A script that performs the verification of condition \eqref{eq:u-u0}. It loads the Fourier coefficients of $u_{0}$, and the corresponding matrix elements of $\mathcal{A}$ for all three solutions and evaluates the condition using predefined parameters $\delta$. 
\item \codett{CAPConstructAcal.m}: A script which constructs the matrix $\mathcal{A}$ according to the procedure outlined in Section~\ref{sec:A}, see also below. The input consists of the Fourier coefficients of $u_0$ and the desired size of the generated matrix. The provided script is configured to analyse $u_0^{(2)}$; other solutions can be treated analogously by adjusting the input data.
\end{itemize}

Each \codett{CAPRun_*.m} script begins with a block of parameter definitions:
\begin{itemize}
\item The frequency via integers \codett{p0} and \codett{q0}, with \codett{Omega0 = (2 p0 + 1)/(2 q0)}, cf. \eqref{eq:omega},
\item The path to the directory \codett{dirIn} where data files reside,
\item Operator bound parameters \codett{K0} and \codett{delta},
\end{itemize}

Input data and computations are performed using rational arithmetic. This guarantees that no rounding errors affect the results. 
Because the arithmetic operations generate rational numbers with lengthy numerators and denominators, and because the total number of operations is substantial (particularly for solutions on branches), the code is designed to be executed in parallel using \textit{Mathematica}'s data parallelism functionality (specifically using \codett{ParallelTable} and \codett{ParallelDo} built-in functions). Accordingly, parallel kernels should be initialized at the beginning of each \codett{CAPRun_*.m} script by calling the \codett{LaunchKernels}.
In contrast, the script \codett{CAPRun_69_40_dist.m} that performs a simple check of \eqref{eq:u-u0} runs in serial.

In addition, we provide the script \codett{CAPConstructAcal.m}, which contains the code used to construct the matrix $\mathcal{A}$ and may be used to independently verify the provided data. The construction, described in Section~\ref{sec:A}, is carried out in high-precision floating-point arithmetic, which is controlled at the software level and thus independent of the underlying hardware architecture. This approach eliminates the influence of rounding errors, which may otherwise vary across machines when standard double-precision is used instead. The resulting matrix is subsequently converted to exact rational numbers.

\subsection{Data}
The included data files provide the rational Fourier coefficients of the function $u_0$, and the corresponding matrix $\mathcal{A}$, for various approximate solutions. File naming follows the convention:
\begin{itemize}
	\item \codett{u0hat_<label>.m}: Fourier coefficients of $u_0$, see \eqref{eq:25.05.22_01}, as a nested list (representing a matrix) with $c_{m,n}$ saved in \codett{u0hat[[m+1,n+1]]},
	\item \codett{Acal_<label>.m}: elements of matrix $\mathcal{A}$, as a nested list with \codett{Acal[[i+1,j+1]]} corresponding to $\mathcal{A}_{ij}$,
\end{itemize}
where \codett{<label>} indicates the solution (e.g., 1st, 2nd, 3rd).
Although the text files store the data as nested lists following the \textit{Wolfram Mathematica} syntax, they can be opened with any standard text editor, and the data can be easily converted to other formats if desired.

The files must be placed in the directory specified by the \codett{dirIn} variable in the associated script. The script \codett{CAPRun_*.m} automatically loads and processes these files as part of the verification procedure.

\subsection{Output}

The primary output of each \codett{CAPRun_*.m} script includes:
\begin{itemize}
\item Each computed operator bound, printed to the screen both in exact rational form and as a machine-precision floating-point number for convenience,
\item The results of testing the inequalities appearing in Theorem~\ref{thm:cap},
\item Values of auxiliary quantities used during the verification process, printed for transparency and reproducibility.
\end{itemize}

No manual post-processing is required; all outputs are self-contained within the run log produced by the script. When run in the command line, as e.g.
\begin{verbatim}
$ math -script CAPRun_69_40_2nd.m | tee 69_40_2nd.log
\end{verbatim}
(assuming that \codett{tee} utility is available) the script prints output both to the terminal and to the file \codett{69_40_2nd.log}. Alternatively, the \codett{CAPRun*.m} file may be opened and executed directly within \textit{Wolfram Mathematica}, with results displayed in the output cells. However, these outputs are not automatically saved, so we recommend using the script-based execution.

The computations are memory-intensive, especially when constructing large symbolic expressions in rational arithmetic. Sufficient memory and parallel capacity are therefore necessary for timely completion.

For computations, we have used a heterogeneous 38-core Beowulf cluster, consisting of 4th to 8th generation Intel Core processors.
For reference, we provide execution times and memory usage\footnote{The simplest example should complete on a single kernel in just under 16 hours.} for the examples included:
\begin{itemize}
\item 1st solution ($M = N = 36$): 45.5 hours, 1290MB of RAM per kernel (48GB total),
\item 2nd solution ($M = N = 38$): 58.2 hours, 1462MB of RAM per kernel (54GB total),
\item 3rd solution ($M = N = 13$): 24.5 minutes, 347.36MB of RAM per kernel (13GB total). 
\end{itemize}

\section{Details of the computations}
\label{sec:DetailsOfTheComputerCode}

This section supplements the overview given in Section~\ref{sec:ComputerAssistance} with explicit formulas and conventions used in the implementation of the computer-assisted proof. We begin by presenting the expressions used to compute the Fourier coefficients of nonlinear terms derived from the approximate solution $u_0$. We then introduce the indexing convention used to transform between two-dimensional mode indices and a one-dimensional representation required for assembling the operator $\mathcal{A}$. Finally, we provide the formula used to compute the operator norm bound of $H_0$ in this notation.

\subsection{Fourier coefficients of nonlinear terms}

Let $u_0$ be the approximate solution given in \eqref{eq:25.05.22_01}. Using standard product identities for trigonometric functions, we compute the Fourier coefficients of $u_0^2$, $u_0^3$, and the action of the multiplication operator $\Lambda_{u_0^2}$ on a basis element $P_{m,n}$ as follows.

\subsubsection{Squared nonlinearity $u_0^2$} Using a discrete convolution of the coefficients $c_{m,n}$ of $u_0$ we find:
\begin{equation}
    \label{eq:25.05.22_02}
    u_0^{2} = \sum_{\substack{0\leq m< 2M\\ 0\leq n< 2N}}d_{m,n} P_{m,n}\,,
\end{equation}
with
\begin{multline}
    \label{eq:25.05.22_03}
    d_{m,n} = \frac{1}{4 (\delta _{0,m}+1) (\delta _{0,n}+1)} \left[ \sum _{m_1=\max (0,m-M)}^{\min (m-1,M-1)} \left(\sum _{n_1=0}^{-n+N-1}
   c_{m-m_1-1,n+n_1} c_{m_1,n_1}\right.\right.
   \\
   \left.+\sum _{n_1=n}^{N-1} c_{m-m_1-1,n_1-n} c_{m_1,n_1}
   -\sum _{n_1=\max (0,n-N)}^{\min (n-1,N-1)} c_{m-m_1-1,n-n_1-1} c_{m_1,n_1}\right)
   \\
   +\sum _{m_1=m}^{M-1} \left(\sum _{n_1=0}^{-n+N-1}
   c_{m_1,n_1} c_{m_1-m,n+n_1}+\sum
   _{n_1=n}^{N-1} c_{m_1,n_1} c_{m_1-m,n_1-n}\right.
   \\
   \left.-\sum _{n_1=\max (0,n-N)}^{\min (n-1,N-1)} c_{m_1,n_1} c_{m_1-m,n-n_1-1}\right)
   \\
   +\sum _{m_1=0}^{-m+M-1} \left(\sum _{n_1=0}^{-n+N-1}
   c_{m_1,n_1} c_{m+m_1,n+n_1}+\sum
   _{n_1=n}^{N-1} c_{m_1,n_1} c_{m+m_1,n_1-n}\right.
   \\
   \left.\left.-\sum
   _{n_1=\max (0,n-N)}^{\min (n-1,N-1)} c_{m_1,n_1} c_{m+m_1,n-n_1-1}\right)\right]\,.
\end{multline}

\subsubsection{Cubic nonlinearity $u_0^3$} Similarly, via convolution of $u_{0}^{2}$ with $u_{0}$ we get:
\begin{equation}
    \label{eq:25.05.22_04}
    u_0^{3} = \sum_{\substack{0\leq m< 3M-1\\ 0\leq n< 3N-1}}f_{m,n} P_{m,n}\,,
\end{equation}
where
\begin{multline}
    \label{eq:25.05.22_05}
    f_{m,n} = \frac{1}{4} \left[\sum _{m_1=\max (0,m-2 M+1)}^{\min (M-1,m)} \left(\sum _{n_1=\max
   (0,n-2 N+1)}^{\min (N-1,n)} d_{m-m_1,n-n_1} c_{m_1,n_1}\right.\right.
   \\
   \left.
   + \sum _{n_1=n}^{N-1} d_{m-m_1,n_1-n} c_{m_1,n_1} - \sum _{n_1=0}^{\min
   (N-1,-n+2 N-2)} d_{m-m_1,n+n_1+1} c_{m_1,n_1} \right)
   \\
   + \sum_{m_1=m}^{M-1} \left(\sum_{n_1=\max (0,n-2 N+1)}^{\min (N-1,n)} d_{m_1-m,n-n_1} c_{m_1,n_1}
   + \sum _{n_1=n}^{N-1} d_{m_1-m,n_1-n}
   \right.
   \\
   \left.
   - \sum _{n_1=0}^{\min(N-1,-n+2 N-2)} d_{m_1-m,n+n_1+1} c_{m_1,n_1}\right)
   \\
   + \sum_{m_1=0}^{\min (M-1,-m+2 M-2)} \left(\sum _{n_1=\max (0,n-2 N+1)}^{\min (N-1,n)}d_{m+m_1+1,n-n_1} c_{m_1,n_1}
   + \sum _{n_1=n}^{N-1} d_{m+m_1+1,n_1-n} c_{m_1,n_1}
   \right.
   \\
   \left.\left.
   - \sum _{n_1=0}^{\min (N-1,-n+2 N-2)} d_{m+m_1+1,n+n_1+1}c_{m_1,n_1}\right)
    \right]
   \,.
\end{multline}

\subsubsection{Action of $\Lambda_{u_0^2}$ on basis functions}
Using \eqref{eq:25.05.22_02} and \eqref{eq:25.05.22_03} we find
\begin{equation}
    \label{eq:25.05.22_06}
    \Lambda_{u_0^2}P_{m,n} = \sum_{\substack{0\leq m_{2}<2M+m\\ 0\leq n_{2}<2N+n}}g^{m,n}_{m_{2},n_{2}} P_{m_{2},n_{2}}\,,
\end{equation}
with
\begin{multline}
    \label{eq:25.05.22_07}
   g^{m,n}_{m_{2},n_{2}}=\frac{1}{4}\left(d_{m-m_2,n-n_2}+d_{m-m_2,n_2-n}-d_{m-m_2,n+n_2+1}
   \right.
   \\
   +d_{m_2-m,n-n_2}+d_{m_2-m,n_2-n}-d_{m_2-m,n+n_2+1}
   \\
   \left.+d_{m+m_2+1,n-n_2}+d_{m+m_2+1,n_2-n}-d_{m+m_2+1,n+n_2+1}\right)\,,
\end{multline}
where we assume $d_{m,n}=0$ whenever the indices fall outside the range defined in \eqref{eq:25.05.22_02}. Alternatively, the coefficients $g^{m,n}_{m_{2},n_{2}}$ can be computed using \eqref{eq:uvP_decompositionMN}-\eqref{eq:uvP_decomposition}, in practice, we opted for using \eqref{eq:25.05.22_06}-\eqref{eq:25.05.22_07} as it led to a slightly cleaner implementation and did not require storing an extra matrix. Recall that by definition the coefficients in \eqref{eq:25.06.03_03} are $\tilde{c}_{\tilde{m},\tilde{n}}=g^{2M-1,2N-1}_{\tilde{m},\tilde{n}}$.

\subsection{Indexing convention}

To efficiently store and manipulate the two-dimensional Fourier coefficients with indices $(m,n)$, we employ a one-dimensional indexing scheme. The map $J(m,n):\ \mathbb{N}\times\mathbb{N}\ni (m,n)\mapsto J\in\mathbb{N}$ is defined by:
\begin{equation}
    \label{eq:25.05.22_08}
    J(m,n) = \max (m,n)^2+\max (m,n)-m+n\,.
\end{equation}
The inverse map $(m,n)(J)$ is given by:
\begin{equation}
    \label{eq:25.05.22_09}
(m,n)(J) = \begin{cases}
      \left(\left\lfloor \sqrt{J}\right\rfloor ,J-\left\lfloor
   \sqrt{J}\right\rfloor ^2\right), & \text{if}\ J-\left\lfloor \sqrt{J}\right\rfloor ^2\leq \left\lceil \sqrt{J}\right\rceil
   ^2-J\,, \\
      \left(\left\lceil \sqrt{J}\right\rceil ^2-J-1,\left\lfloor
   \sqrt{J}\right\rfloor \right), & \text{otherwise}\,.
    \end{cases}
\end{equation}

This indexing scheme is used consistently in the construction of the matrix $\mathcal{A}\in\mathbb{R}^{J_{\mu}\times J_{\mu}}$, where  $J_{\mu}\equiv \mu^{2}$ (we consider truncations with $\nu=\mu$, cf. Section~\ref{sec:A}), and in the evaluation of operator bounds, see below.

\subsection{Evaluation of bound on $\Vert H_{0}\Vert$}

Given the value of $C$ computed according to \eqref{eq:25.05.20_04}, we determine suitable values of $\tilde{M}$ and $\tilde{N}$  by requiring that the last two terms in \eqref{eq:25.05.20_03} are each strictly less than one. In practice, we set both $\tilde{M}$ and $\tilde{N}$ equal to the smallest value satisfying this condition for both terms.

The first term in \eqref{eq:25.05.20_03} involves a maximum of two operator norms: one for the truncated operator $A=\mathcal{A}$, and the other for $A=I$ (see Section~\ref{sec:ComputerAssistance}). Both cases are expressed and evaluated using the one-dimensional indexing scheme introduced above. In the following, we restrict ourselves to approximate solutions \eqref{eq:25.05.22_01} with diagonal truncations $N=M$.

We begin by rewriting \eqref{eq:25.05.22_06} in terms of the one-dimensional index $K$:
\begin{equation}
    \label{eq:25.06.02_01}
    \Lambda_{u_{0}^2}P_{K} = \sum_{K_{2}=0}^{K_{2}^{\max}(K)-1}g^{K}_{K_{2}}P_{K_{2}}\,.
\end{equation}
where $K_{2}^{\max}(K)$ corresponds to the upper limits in \eqref{eq:25.05.22_06}, and we use the shorthand notation $P_{K}\equiv P_{m(K),n(K)}$, and $g^{K}_{K_{2}}\equiv g^{m(K),n(K)}_{m_{2}(K_{2}),n_{2}(K_{2})}$. Next, recalling the definition \eqref{eq:H0}, we compute $H_{0}P_{J}$ for $0\leq J<J_{\mu}$, i.e. for $0\leq m,n<\mu$, as follows
\begin{equation}
    \label{eq:25.06.02_02}
    H_{0}P_{J} = -3\sum_{K_{2}=0}^{J_{\mu}-1}\sum_{K=0}^{K_{2}^{\max}(K)-1}\mathcal{A}_{J,K_{2}}g^{K_{2}}_{K}L_{\Omega}^{-1}P_{K}+P_{J}- \sum_{K={0}}^{J_{m}-1}\mathcal{A}_{J,K}P_{K}\,.
\end{equation}
We rearrange the double sum by extending the upper bound to a fixed maximum and interchanging the summation
\begin{equation}
    \label{eq:25.06.02_03}
    H_{0}P_{J} = -3\sum_{K=0}^{K_{\max}-1}\sum_{K_{2}=0}^{J_{\mu}-1}\mathcal{A}_{J,K_{2}}g^{K_{2}}_{K}L_{\Omega}^{-1}P_{K}+P_{J}- \sum_{K={0}}^{J_{m}-1}\mathcal{A}_{J,K}P_{K}\,,
\end{equation}
where $K_{\max}=(2M+J_{\mu}^{1/2}-1)^{2}$, and we use the convention that $g^{K_{2}}_{K}=0$ whenever $K_{2}\geq K_{2}^{\max}(K)$. Thus
\begin{multline}
    \label{eq:25.06.02_04}
    \Vert H_{0}P_{J} \Vert = \frac{1}{\rho(J)}\left(\sum_{K=0}^{J_{m}-1} \rho(K) \left| -3\sum_{K_{2}=0}^{J_{m}-1}\mathcal{A}_{J,K_{2}}g^{K_{2}}_{K}L_{\Omega}^{-1}(K) + \delta_{JK}- \mathcal{A}_{J,K} \right| 
    \right.
    \\
    \left.
    + \sum_{K=J_{m}}^{K_{\max}-1}\rho(K)\left| -3\sum_{K_{2}=0}^{J_{m}-1}\mathcal{A}_{J,K_{2}}g^{K_{2}}_{K}L_{\Omega}^{-1}(K) \right|
    \right)\,, \quad 0\leq J<J_{\mu}\,,
\end{multline}
where $\rho(K):=\rhoOne^{2m(K)+1}\rhoTwo^{2n(K)+1}$, and 
\begin{equation}
    L_{\Omega}^{-1}(K):=\frac{1}{(2n(K)+1)^2-\Omega^2 (2m(K)+1)^2}\,,
\end{equation}
with $(m,n)(K)$ defined in \eqref{eq:25.05.22_09}. For $m$ and $n$ in $[\mu, \tilde{M}-1]$, corresponding to $J_{\mu}\leq J<\widetilde{J}-1$, with $\widetilde{J}=\tilde{M}^{2}$, we have:
\begin{equation}
	\label{eq:25.06.02_05}
	\Vert H_{0} P_{J}\Vert = \frac{1}{\rho(J)} \sum_{J_{2}=0}^{J_{\max}(J)-1}\rho(J_{2})\left| -3g^{J}_{J_{2}}L_{\Omega}^{-1}(J_{2}) \right|\,, \quad J_{\mu}\leq J<\widetilde{J}\,.
\end{equation}
with $J_{\max}(J) = \left[2M+\max(m(J),n(J))\right]^{2}$.
Thus, the contribution of the first term in \eqref{eq:25.05.20_03} is given by the maximum over the expressions \eqref{eq:25.06.02_04} and \eqref{eq:25.06.02_05}.
This completes the computation of the upper bound for the operator norm $\Vert H_{0}\Vert$.

We remark that no attempt was made to optimize the weights in the norm; in all computations, we simply take $\rho_\tau = \rho_x = 1 + 10^{-20}$.

\medskip

\medskip

\medskip

\printbibliography

\end{document}